\documentclass[12pt]{article}
\usepackage{graphicx}
\usepackage{amsmath}
\usepackage{amsthm}
\usepackage{amssymb}
\usepackage{latexsym}
\usepackage{setspace}
\theoremstyle{plain}
\newtheorem{theorem}{Theorem}[section]
\newtheorem{lemma}[theorem]{Lemma}

\theoremstyle{definition}

\theoremstyle{remark}

\def\RR{{\bf R}}
\def\ZZ{{\bf Z}}

\begin{document}
\begin{center}
{\Large \bf Stability in a many-to-one job market with general \\
increasing functions}\\[12mm]

{\bf Yasir ALI$^a$\footnote{corresponding author}, Baqar ALI$^b$ }\\[3mm]
$^a$ \footnotesize{National University of Sciences and Technology,
College of Electrical and Mechanical Engineering, Peshawar Road,
Rawalpindi, Pakistan}\\
$^b$ \footnotesize{Riphah International University, I-14, Islamabad,
Pakistan}
\end{center}
{\makeatletter{\renewcommand*{\@makefnmark}{} \footnotetext{E-mail
address:  {\tt cyasirali@gmail.com} (Y. Ali). \makeatother}}
\begin{abstract}
We consider an occupation market in which preferences of members are
treated as non linear general increasing functions. The arrangement
of members is separated into two non over-lapping sets; set of
workers and set of firms. We consider that firms have vacant posts.
Every worker needs a job and firms have opportunity to contract more
than one workers. A worker can work for just in at most one firm. We
demonstrate the existence of pairwise stability for such a business
sector. Our model is the augmentation of the Ali and Farooq
\cite{ya2010} model by considering non linear valuations and bounded
side payments.
\end{abstract}
\begin{center}
\textbf{{keyword:}} Stable matching, many-to-one matching,
indivisible goods, increasing valuations
\end{center}

\section{Introduction}
Over the last few decades, a lot of research has been carried out in
the two-sided matching problem by a large number of scholars. In a
two-sided matching problem, the set of members is divided into two
non over-lapping sets, where each member has a list of preferences
which has the members of another set. Basically the matching is a
mapping from one set of members to another set. More preciously the
goal of two sided matching is to formulate the  partnership between
the members of two sets. This matching must be stable as well.

A matching $X$ is called stable matching if there is no blocking
pair and all members currently matched are mutually acceptable. The
concept of two-sided stable matching was introduced by Gale and
Shapley \cite{GS} in their famous article ``Marriage problem and
College admission". Another fundamental article in this area is due
to Shapley and Shubik \cite{SS}, that is known as assignment game.
Gale and Shapley's marriage model and Shapley and Shubik's
assignment game have been broadly studied and a number of extensions
of these can be found in the literature. Many researchers studies
the extension based one number of partners in matching. More
specially, these models that involve many-to-many and many-to-one
matchings. The other important direction of research is to look for
the most general valuation function that includes these two basic
models, Gale and Shapley \cite{GS} and Shapley and Shubik \cite{SS}
and well known generalization of these models like \cite{EriKar},
Sotomayor \cite{Sot2000} and \cite{Far2008}, as special cases. In
this context, Farooq \cite{Far2008} and Ali and Farooq \cite{af2011}
are particularly important. In Farooq \cite{Far2008}, valuation
function is considered as linear function of money where money in
\cite{Far2008} is taken to be continuous variable. Eriksson and
Karlander \cite{EriKar} and Sotomayor \cite{Sot2000} both generalize
the marriage model and assignment game. Both \cite{EriKar} and
\cite{Sot2000} are the special cases of Farooq \cite{Far2008}.
Whereas in Ali and Farooq \cite{af2011}, many-to-one matching is
considered with preferences of the players are represented as a
linear function of money. Thus the models \cite{GS},\cite{SS},
\cite{Far2008} become the special cases of Ali and Farooq
\cite{af2011}. We can find many articles with many-to-many and
many-to-one matching. Many-to-one matching models due to Pycia
\cite{P05} and Echenique and Yenmez \cite{Eco} consider two
different matching models. For the models when there
complementarities and peer effects, Pycia \cite{P05} provides some
important condition for the stability. Many-to-one matching model
\cite{Eco} provides solution for the matching market where the
preferences are over the colleagues.

Here we highlight some related models that generalize valuation
functions. The linear valuation are presented in Farooq
\cite{Far2008} and Ali and Farooq \cite{af2011} with continuous
variable, but the valuation defined on discrete domain, is presented
in Ali and Farooq \cite{ya2010}. The model \cite{ya2010} addresses
to the situations where the side payments are indivisible.

In Ali \cite{ya2015}, the set of members is classified into two non
over-lapping sets: a set of workers and a set of firms. The set of
workers is further classified into subsets, that represent different
categories in everyday life. Firms are bounded to hire more than one
worker from any category. A worker can work in only one category for
at most one firm. This is a novel idea that is different from simple
many-to-one matching that we have considered in our model.

Recently, Ali and Javaid \cite{ay2014} presented a model which
involves general increasing functions. This model \cite{ay2014} is
generalization of \cite{ya2010}.  In this work we have considered a
generalization of Ali and Javaid \cite{ay2014} model. In our model a
job market is considered and the valuation of members are treated as
general increasing functions which may be nonlinear with discrete
side payments and the matching is many-to-one. The general nonlinear
function for the valuations are $f_{ij}(z)$ and $f_{ji}(-z)$ for
each $(i, j) \in E$. A firm can hire more than one worker and vacant
positions are available for worker's job. The number of workers that
a firm can hire is known as the quota of that firm. Workers can work
for only at most one firm in this model.

In Section 2 describes of our model briefly. Section 3 gives the
chronological mechanisms for worker and firm. Section 4 describes
the outcome and pairwise stability for our model. We devise an
algorithm which finds a stable outcome in our model in Section 5. In
Section 6, we discuss the main result of our model.
\section{The Mathematical Model}\label{MM}
We represent a model of matching market for job allocation of
workers and firms. Let us consider a matching market in which there
are two non over-lapping sets of worker and firms. The requirement
of each firm is the workers and the requirement of each worker is a
job in a firm. A worker gives a financial benefit in form of revenue
to the firm and firm pays a reward (money) to its workers, called
\textit{salary} of the worker. Firms hire workers to fill up their
vacant positions.

Let us express our model mathematically, we take two sets $F$, a set
of firms and $W$, a set of workers. Here $E=F\times W$ shows all
possible pairs of firms and workers. Each firm hires some workers
from all available workers. The maximum number of workers that a
firm $j$ requires is denoted by $\mu(j)$ and is known as the quota
of a firm $j$. A firm can not hire workers more than $\mu(j)$. Note
that each worker can work for only at most one firm.

To increase or decrease the salary depends on the worker and firm
consultancy. It is assumed that the salaries are bounded in this
model i.e $\forall (i,j)\in E$, $ a_{ij}$ and $ b_{ij}$ shows the
lower and upper bounds on salary, where $ a$, $
b\in\ZZ^{E}$\footnote{The notation $\ZZ$ stand for set of integers
and notation $\RR$ stand for set of real numbers. The notation
$\ZZ^E$ stands for integer lattice whose  points are indexed by
$E$.}. Let $p=(p_{ij}\mid(i,j)\in E)\in \ZZ^{E}$ shows the
\textit{salary vector} for all $(i, j) \in E$. It is
\textit{feasible} if $ a\leq p \leq b$\footnote{For any $x,y\in
\ZZ$, we define $[x,y]_{\ZZ} = \{a\in \ZZ \mid\ x\leq a \leq y\}$.}.
We consider the preferences of players as the strictly general
increasing functions, known as generalized increasing valuations.
The valuation stand for estimation of real property or some asset.

For all $(i,j)\in E$, define $f_{ij}(z)$ and $f_{ji}(-z)$ from $\ZZ$
into $\RR$, where{ $z\in \ZZ$}. $f_{ij}(z)$ stands for the valuation
of the worker $i$, at a salary $z$ from a firm $j$, when worker $i$
joins firm $j$ (that is, matched with $j$). Similarly, $f_{ji}(-z)$
stands for the valuation of the firm $j$ when it hires worker $i$
and pays it salary $z$. It is important to note here that valuations
play a significant role in establishing the stability of matching.
\section{The Firm-Worker Chronological Mechanism}
\label{sec-FWCM}
Using valuations mentioned above, a comparison of
the players of both sets can be made. A worker $i$ will
\textit{prefer} the firm $j$ to the firm $j^*$ at salary $z,z^*\in
\ZZ$, if $f_{ij}(z)> f_{ij^*}(z^*)$. If $f_{ij}(z) = f_{ij^*}(z^*)$
than $i\in W$ is \textit{indifferent} between firms $j$ and $j^*$ at
salary $z,z^*\in \ZZ$ .

Let us define the terms ``indifferent'' and``prefer'' for a firm in
the same sense. If a worker is agreed for a job in a firm then it
means the firm is \textit{acceptable} to that worker. Similarly, if
a firm has willing to hire a worker then it means the worker is
acceptable to that firm. By $f_{ij}(z)\geq 0$, represents that $j$
is \textit{acceptable} to $i$ at salary $z\in \ZZ$ and
$f_{ji}(-z)\geq 0$ represents $i$ is acceptable to $j$ at salary
$z$.
\section{Outcome and Pairwise Stability in Matching}
\label{sec-OPSM} In the two-sided stable matching theory, where firm
and worker involves money plays an important role. In this section,
we describe the characteristic of an outcome for which it would be
stable.

Let $E$ is the set of all possible firm-worker pairs. A subset $X$
of a set $E$ is called matching if every member appear atmost once
in $X$. ``A matching is called pairwise stable if it is not blocked
by any worker-firm pair and all members of matched pairs are
mutually acceptable."
We define $S=(S_j\mid j\in F)$, where $S_j$ is given below,
\begin{equation}\label{S}
S_j=\{i \in W \mid i \mbox{ and } j \mbox{ are matched} \}.
\end{equation}
If $S_j=\emptyset$ $\forall j\in F$, we means all positions are
vacant in the firm $j$. If $i\in S_j$ we means that a firm $j$
\textit{hires} a worker $i$ .

A set $X=\{(S_j,j)\mid  j\in F\}$ is known as \textit{job
allocation} if
\begin{description}
\item[\rm{(i)}]
$|S_j|\leq \mu(j)$  $\forall j\in F$.
\item[\rm{(ii)}]
$S_j \cap S_{j^*} = \emptyset$ $\forall j,j^*\in F$ with $j\neq
j^*$.
\end{description}
First condition may be considered as quota requirement for all firms
and second condition tells that no worker can work for more than one
firm. This may be considered as quota condition for workers. For any
feasible salary $p_{ij}$, for all $(i,j)\in E$ define $q\in \RR^{W}$
as follows:
\begin{eqnarray}
q_i & = & \left \{
\begin{array}{ll}
f_{ij}(p_{ij}) & \mbox{if $i\in S_j$ for any $j\in F$}\\
0                & \mbox{otherwise}
\end{array}
\right.
\begin{array}{l}
\quad (\forall i\in W).
\end{array}\label{payoff1}
\end{eqnarray}

We have $r=(r_j | j\in F) \in \RR^F$, is defined as
\begin{eqnarray}
r_j &=& \left \{
\begin{array}{ll}
\min \{f_{ji}(-p_{ij})\mid i\in S_j\}  & \mbox{if $|S_j|=\mu(j)$}\\
0                  & \mbox{otherwise}
\end{array}
\right.
\begin{array}{l}
\quad (\forall i\in W),
\end{array}\label{payoff2}
\end{eqnarray}
where minimum over a null set is defined to be $0$.

If $X$ is a job allocation than $(X;p,q,r)$ is known as an outcome,
$p$ is feasible salary vector, and $q$ and $r$ are given by
\eqref{payoff1} and \eqref{payoff2}, respectively.

For convenience, we write that $S_j \in X$ (or $j\in X$), it always
means that $(S_j, j)\in X$ and by $(i,j)\in X$, we always mean that
$i\in S_j$.

An outcome will be \textit{blocked} by a worker-firm pair in which
firm and worker are not matched to each other but both of these
unmatched members prefer each other to their current partners. Note
here that there may be some members who are not matched to any
member from the opposite set, such a member is called self matched.
In mathematical language, the outcome $(X;p,q,r)$ has a blocking
pair $(i,j)\in E$ if $\exists$ $\theta \in \ZZ$ with $a_{ij}\leq
\theta \leq b_{ij}$ such that $i\notin S_j$ and
$f_{ij}(\theta)>q_i$, $f_{ji}(-\theta)>r_j$.

\subsection{Pairwise Stability}
Here for outcome $(X;p,q,r)$, the pairwise stability is defined as:
\begin{description}
\item[(ps1)]
 $f_{ij}(p_{ij}) \geq 0$ and
$f_{ji}(-p_{ij}) \geq 0$  $\forall (i,j)\in X$.
\item[(ps2)]
$f_{ij}(\theta)\leq q_i$ or $f_{ji}(-\theta) \leq r_j$, $\forall$
$\theta \in \ZZ $ with $ {a}_{ij} \leq \theta \leq {b}_{ij}$ and
$\forall (i,j)\in E$.
\end{description}
The (ps1) is showing the mutual acceptability of matched pairs, and
(ps2) shows that the outcome will not be blocked by any of the
pairs.
\section{Existence of a Stable Outcome in This Model}\label{sec-ESOM}
In this part, it will be shown that the pairwise stability always
exist for the model given in previous segment. We develop an
algorithm to show the existence of pairwise stable outcome.
Initially, the highest feasible salary is set out in the algorithm
in such a way that firms are acceptable to workers at that salary.
It is important to note that at this initial value of the salary
workers may not be acceptable to some firms. In other words, we can
say that at initially fixed salary firms and workers may not be
mutually acceptable. Then the mutually acceptable pairs of workers
and firms are found. After this workers are to be engaged to the
firms that the workers prefer most. At this point firms accept
proposals by considering their quota and gain. Due to these
constraints there may be some rejections. These rejections induce
the modification in salary vector. The salary is adjusted for a pair
if worker is rejected by the firm he or she prefers most. Salary is
adjusted in every iteration conserving the feasibility, till the
pairwise stability is obtained. The algorithm will stop if there are
no rejections. Lastly, we will show that when the algorithm
terminates, the output is a stable matching.

Now $\forall (i,j)\in E$, $p_{ij}\in \ZZ$ is given as below:
\begin{equation}\label{pij}
p_{ij} = \left \{
\begin{array}{ll}
 {b}_{ij} & \mbox{if } f_{ji}(- {b}_{ij}) \geq 0 \mbox{ and }\\
\max\left\{a_{ij}, \left\lfloor{-f^{-1}_{ji}(0)}\right\rfloor
\right\} & \mbox{ elsewhere}.
\end{array}
\right.
\begin{array}{l}
\quad (\forall (i,j)\in E).
\end{array}\\
\end{equation}
Equation \eqref{pij} guarantees that the salary vector $p$ is a
feasible. Also, note that $p$, defined by \eqref{pij}, is the
maximum integer in $[a, b]$ for which $f_{ji}(-p_{ij})\geq 0$
$\forall (i,j)\in E$.

Before presenting the algorithm in mathematical form, let us define
some subsets of set $E$ which are useful to obtain a matching $X$
that satisfies the condition (ps1). First, the subset $W_{0}$ and
$F_{0}$ of set $E$, containing those worker-firm pairs that are not
mutually acceptable are defined as:
\begin{equation}\label{X0}
W_0 = \{ (i,j)\in E \mid f_{ij}(p_{ij})< 0\} .
\end{equation}
\begin{equation}\label{Y0}
F_0 = \{ (i,j)\in E \mid f_{ji}(-p_{ij})< 0\}.
\end{equation}
$W_0$ is the set containing those pairs where worker is not willing
to work in firm and $F_0$ is the set containing those pairs where
firm is not willing to hire the worker at $p$ given by \eqref{pij}.

Next, we define $\widetilde E$ by
\begin{equation}\label{tile}
\widetilde E = E\setminus\{W_0 \cup F_0\},
\end{equation}
or
\begin{equation}\label{til-E}
\widetilde E=\{(i,j)\in E\mid f_{ij}(p_{ij})\geq 0\mbox{ and
}f_{ji}(-p_{ji})\geq0\}.
\end{equation}
$\widetilde E$ contains the set of mutually acceptable players.

For each $i\in W$ define $\tilde q_i$ as follows:
\begin{equation}\label{til-q}
\tilde q_i = \max \{f_{ij}(p_{ij})\mid (i,j)\in \widetilde E\} \quad
                (\forall i \in W).
\end{equation}
In equation \eqref{til-q} $\tilde q_i$ represents the valuation of
most preferred firm for $i \in W$. Also, define $\widetilde E_W$ as
follows:
\begin{equation}\label{til-EW}
\widetilde E_W=\{(i,j)\in \widetilde E\mid f_{ij}(p_{ij})=\tilde
q_i\}.
\end{equation}

Equation $\eqref{til-EW}$ shows that $\widetilde E_W$ contains those
worker-firm pairs where firm is most preferred for the worker. Since
$\widetilde E_W \subseteq \tilde E$ this means that these pairs are
acceptable. Initially consider vector $r = \textbf{0}$ and the
subset $\widehat{E}_W$ of $\widetilde{E}_W$ is given by:

\begin{equation}\label{hatEW}
\widehat{E}_W = \{ (i,j) \in \widetilde{E}_W \mid
f_{ji}(-p_{ij})\geq r_j\}.
\end{equation}
It can be noted that $\widehat{E}_W = \widetilde{E}_W$ where $r=0$.
But in the further iterations of algorithm $\widehat{E}_W$ may
become a proper subset of $\widetilde{E}_W$.

At the start of the algorithm, there is no matching $X$, so
$\widetilde {F} = \emptyset$, where $\widetilde {F}$ shows the set
of matched firms in $X$, and defined by

\begin{equation}\label{til-F}
\widetilde F =\{j\in F | j \mbox{ is matched in } X\}.
\end{equation}

Now, we find a job allocation $X =\{(S_j,j)\mid j\in F\}$, in the
bipartite graph $(W,F;\widetilde E_W)$ that satisfies the conditions
given below:

\begin{eqnarray}
&& X \mbox{ matches all members of } \widetilde F,
\label{X1}\\
&&\nonumber X \mbox{ minimizes $\left| |S_j|-\mu(j) \right|$, among the matchings}\\
&&\mbox{ that satisfy }  \eqref{X1}. \label{X2}\\
&&\nonumber X \mbox{ maximizes }  \sum\limits_{(i,j)\in X}f_{ji}(-p_{ij}) \mbox{ among the matchings}\\
&&\mbox{ that satisfy } \eqref{X1} \mbox{ and }
\eqref{X2}.\label{X3}
\end{eqnarray}

Initially $\widetilde {F} = \emptyset$, which follows that any
matching satisfies $\eqref{X1}$. The outcome $(X; p, q, r)$ up till
these steps clearly satisfies the (ps1). For the satisfaction of
(ps2), next we define a set $U$ having all mutually acceptable
worker-firm pairs. Also firm is most preferred for the worker but it
reject the workers and consequently is unmatched in the matching $X$
by:
\begin{equation}\label{U}
U=\{(i,j)\in \widetilde E_W\mid i\notin S_j\}.
\end{equation}
$U$ has worker-firm pairs in which worker is not matched to his or
her most preferred firm. As the set $\widetilde E_W$ has such
mutually acceptable pairs and the firm is most preferred for the
worker out of all firms.

If $U = \emptyset$, then do not modify the salary vector $p$ but if
$U \neq \emptyset$, then salary vector will be modified at each
iteration of STEP[3] by satisfying (ps1). New salary vector $\tilde
p$ must be feasible, that is $a_{ij} \leq \tilde p_{ij} \leq
b_{ij}$.

As here general non linear strictly increasing functions are
considered to represent the valuations, thus to modify salary vector
$p$, a real number $m_{ij}^* \in \RR^{{+}{+}}$ can be found for each
$(i,j)\in U$, such that
\begin{equation}\label{m1}
f_{ji}(-(p_{ij}-m_{ij}^*)) = r_j.
\end{equation}
As we have considered discrete salaries, so we define an integer
$m_{ij}$ $\forall (i,j)\in U$ as follows:
\begin{equation}\label{mij}
m_{ij}= \max \left\{1, \lceil m_{ij}^{*}\rceil \right\}\qquad
\forall (i,j)\in U.
\end{equation}
The number $m_{ij}$ $\forall (i,j)\in U$ is an integer such that
\begin{equation}\label{m2}
f_{ji}(-(p_{ij}-m_{ij})) \geq r_j,
\end{equation}
Note that $\forall (i,j)\in U$ $p_{ij}- m_{ij}$ is an integer and
$m_{ij}$ is the minimum positive integer, which satisfies the
inequality given by the equation $\eqref{m2}$.

Here the integer $m_{ij}$ for all $(i,j)\in U$ helps us to find the
new salary vector such that condition (ps2) is also satisfied. Now
we define a subset $L$ of $U$ that has the pairs from the $U$ for
which modified salary will not be feasible.

\begin{equation}\label{tilL}
L=\{(i,j)\in U\mid p_{ij}-m_{ij}< a_{ij}\}.
\end{equation}

The modified salary vector $\tilde p$ must also be feasible and is
given as:
\begin{equation}\label{til-p}
\tilde p_{ij}:=\left \{
\begin{array}{ll}
\max \{a_{ij}, p_{ij}-m_{ij}\} & \mbox{if } (i,j)\in U\\
p_{ij} & \mbox{otherwise}
\end{array}\right.
\begin{array}{l}
\quad (i,j)\in E.
\end{array}
\end{equation}
We also define a subset $\widetilde W_0$ of $U$ by:
\begin{equation}\label{U2}
\widetilde W_0:=\{(i,j)\in U \mid  f_{ij}(\tilde p_{ij})< 0\}.
\end{equation}
In the algorithm the modified salary vector will always decrease and
the size of matching $X$ will increase. Also, the participants can
change their preferences according to new salary vector. Now, we
suggest the algorithm.

\subsection*{Job Allocation Algorithm}
\begin{description}
\item STEP[0a]:
First define $p$, $W_0$, $F_0$ $\widetilde E$, $\tilde q$,
$\widetilde E_W$ using equation $\eqref{pij}$-$\eqref{tile}$,
$\eqref{til-q}$, and $\eqref{til-EW}$, respectively. Set
$X=\emptyset$ which gives $r=0$, $q=0$ and $\widetilde F=\emptyset$
by $\eqref{payoff2}$, $\eqref{payoff1}$ and $\eqref{til-F}$. Also
define $\widehat{E}_W$ by $\eqref{hatEW}$.
\item STEP[0b]:
Form the bipartite graph $(W, F, \widehat{E}_W)$ find a matching
that satisfies $\eqref{X1}$ to $\eqref{X3}$. Define $U$ by
$\eqref{U}$ and update $S$ and $r$ by $\eqref{S}$ and
$\eqref{payoff2}$ respectively.
\item STEP[1]:
If $U \neq \emptyset$ go to STEP[2], otherwise define $q$ by
$\eqref{payoff1}$ and stop.
\item STEP[2]:
To update $\tilde p$ compute $m_{ij}$ for all $(i,j)\in U$ by
\eqref{mij} and update $\tilde p$ by \eqref{til-p} and set $p=
\tilde p$.
\item STEP[3]:
Define $L$ by \eqref{tilL} and $\widetilde W_0$ by $\eqref{U2}$.
Modify $W_0: W_0 \cup \widetilde W_0$, and $F_0: F_0 \cup L$ and
update $\widetilde E$ by $\eqref{til-E}$. Modify $\tilde q$,
$\widetilde E_W$ and $\widehat{E}_W$ by $\eqref{til-q}$,
$\eqref{til-EW}$ and $\eqref{hatEW}$, respectively.


\item STEP[4]:
Form the bipartite graph $(W, F, \widehat{E}_W)$ find a matching
that satisfies $\eqref{X1}$ to $\eqref{X3}$. Define $U$ by
$\eqref{U}$ and update $S$ and $r$ by $\eqref{S}$ and
$\eqref{payoff2}$ respectively. Go to STEP[1].
\end{description}
\section{The Main Results}\label{sec-MR}
In this segment we will express that we developed a model by takeing
the preferences of members as non linear strictly increasing general
function. Also this is many-to-one matching model and money is taken
as discrete variable. Is this section, we will prove some important
lemmas and theorems. We will put prefixes $(old)\star$ and
$(new)\star$ before and after the updating the integers/sets/vectors
in any iteration of algorithm. The main result is the Lemma
\ref{max-intg} that we prove here using the assumption given in
equation \eqref{mij}. The proof of Lemma \ref{re-la02} is the direct
consequence of Lemma \ref{max-intg}.

\begin{lemma}\label{max-intg}
For each $(i,j)\in U$, $f_{ji}(-\tilde p_{ij})\geq r_j$ holds at
STEP[2] in each iteration of the algorithm. Also for any $(i,j)\in
U$, $p_{ij}- m_{ij}$ is the maximum integer, whenever
$f_{ji}(-(p_{ij}-m_{ij}))> r_j$.
\end{lemma}
\begin{proof}

From equation $\eqref{til-p}$ we have $\forall (i,j)\in U$
\begin{eqnarray*}
\tilde p_{ij}= \max{\{a_{ij},  p_{ij}-m_{ij}}\}.
\end{eqnarray*}
Without loss of generality we assume that $p_{ij}-m_{ij}>a_{ij}$,
that is, $\tilde p_{ij}=p_{ij}-m_{ij}$. Then by the definition of
$m^*_{ij}$ given in \eqref{m1}, we have
\[
\tilde p_{ij}= p_{ij}-m_{ij} \leq p_{ij}-m_{ij}^{*}.
\]
We can easily write it as $-\tilde p_{ij}= -(p_{ij}-m_{ij}) \geq
-(p_{ij}-m_{ij}^{*})$. Now by the nature of valuation functions we
obtain the following relation
\[
f_{ji}(-\tilde p_{ij})= f_{ji}\left(-(p_{ij}-m_{ij})\right) \geq
f_{ji}\left(-(p_{ij}-m_{ij}^{*})\right)=r_j \qquad \mbox{ by }
\eqref{m1}.
\]
Thus we have
\[
f_{ji}( -\tilde p_{ij}) \geq r_j.
\]
This completes the proof of first part.

For the second part of the lemma, suppose that $f_{ji}(-(
p_{ij}-m_{ij})) > r_{j}$, on contrary suppose that $p_{ij}-m_{ij}$
is not maximum for which the above inequality holds. This means that
there exists another integer $p_{ij}-m_{ij}^{'}$ such that
\begin{equation}\label{m'}
p_{ij}-m_{ij} < p_{ij}-m_{ij}^{'}, \qquad \mbox{where $m_{ij}^{'}
\in \ZZ^{++}$},
\end{equation}
and
\begin{equation}\label{fm'}
f_{ji}(-(p_{ij}-m_{ij}^{'}))>r_j.
\end{equation}
It follows equation \eqref{m'} that $m_{ij} > m_{ij}^{'}$ this means
that
\begin{eqnarray*}
m_{ij} = max \left\{1, \lceil m_{ij}^{*}\rceil\right\} > m_{ij}^{'}.
\end{eqnarray*}
Suppose $m_{ij}= 1$ then $m_{ij}^{'} < 1$, which is not possible as
$m^{'}_{ij}$ must be an integer. Now suppose that
\begin{eqnarray*}
m_{ij} = \lceil m_{ij}^{*}\rceil > m_{ij}^{'}.
\end{eqnarray*}
By definition of ceiling function
\begin{eqnarray*}
m_{ij} \geq m_{ij}^{*} > m_{ij}^{'},
\end{eqnarray*}
which implies that
\begin{eqnarray*}
-( p_{ij}-m_{ij}^{*}) > -( p_{ij}-m_{ij}^{'}),\\
f_{ji}(-(p_{ij}-m_{ij}^{*})) > f_{ji}(-(p_{ij}-m_{ij}^{'})).
\end{eqnarray*}
Using equation \eqref{m1}, we have
\begin{eqnarray*}
f_{ji}(-(p_{ij}-m_{ij}^{'})) < r_j,
\end{eqnarray*}
which is contradiction to $\eqref{fm'}$. Hence $p_{ij}- m_{ij}$ is
the maximum integer.
\end{proof}

The base of our model is Lemma \ref{max-intg} that allows us to
calculate a new salary vector by preserving the of mutually
acceptability and feasibility condition for the salary vector.

\begin{lemma}\label{re-la02}:

At STEP[3] $\tilde p_{ij}$ is feasible and $f_{ji}(-\tilde p_{ij})
\leq (old)r_j$ for each $(i,j)\in L$.
\end{lemma}

\begin{proof}
To show the feasibility of $\tilde p_{ij}$, $\forall (i, j) \in L$
it would be enough to show that $\tilde p_{ij} = a_{ij}$, $\forall
(i, j) \in L$. As $L \subseteq U$ and for each $ (i, j) \in U$, we
have from $\eqref{til-p}$
\begin{eqnarray*}
\tilde p_{ij}=  max{\{a_{ij},  p_{ij}-m_{ij}}\}.
\end{eqnarray*}
By equation $\eqref{tilL}$ we know $\forall (i, j) \in L$,
\begin{eqnarray*}
p_{ij}-m_{ij} < a_{ij},
\end{eqnarray*}
So by equation $\eqref{tilL}$ and $\eqref{til-p}$, and fact that $L
\subseteq U$, it is obvious that $\tilde p_{ij} = a_{ij}$, $\forall
(i, j) \in L$.

  By first part of the lemma we know that
\begin{eqnarray*}
p_{ij}-m_{ij} < a_{ij} =\tilde p_{ij},\\
-\tilde p_{ij}< -(p_{ij}-m_{ij}),\\
f_{ji}(-\tilde p_{ij})<f_{ji}(- (p_{ij}-m_{ij})).
\end{eqnarray*}
We know by lemma \ref{max-intg}, that $f_{ji}(-( p_{ij}-m_{ij}))
\geq r_j$. Result is trivial for $f_{ji}(-( p_{ij}-m_{ij})) = r_j$.
The inequality is true for a maximum integer $p_{ij}-m_{ij}$ by
lemma \ref{max-intg}. As $p_{ij}-m_{ij} < \tilde p_{ij}$ that is why
$f_{ji}(-\tilde p_{ij}) \leq (old)r_j$. Hence the result is proved.
\end{proof}

The next lemma shows that we can find a matching in the bipartite
graph $(W, F; \widehat E_W)$ which satisfies the equations
$\eqref{X1}$ to \eqref{X3}.
\begin{lemma}\label{existence}
At STEP[4] of algorithm, we can always find a matching $X$
satisfying the equations \eqref{X1} to \eqref{X3}
\end{lemma}

\begin{proof}
If it is proved that $(old)X \subseteq (old)\hat E_W$ in each
iteration of the STEP[4]. It completes the proof. We update the
salary vector $p$ and $\tilde E$ in each iteration at STEP[2] and
STEP[3] by equation $\eqref{pij}$ and $\eqref{tile}$. It is clear
from these two equations that these modifications are only element
and subsets of $U$. Since $U \cap (old)X = \emptyset$. Therefore it
implies that $(old)X \subseteq (old)\hat E_W$.

%

\end{proof}

\begin{lemma}\label{la2}
The properties:
\begin{description}
\item[(i)]
For each element of $U\setminus\{L\cup \widetilde
W_0\}\not=\emptyset$, salary $p$ decreases at STEP[3], otherwise
salary vector remains the same.

\item[(ii)]
For each element of ${L\cup \widetilde W_0}$, at STEP[3],
$\widetilde E$ reduces otherwise remains same.

\item[(iii)]

The vector $r$ increases or remains the same.

\end{description}
\end{lemma}

\begin{proof}
(i) Initially the salary vector $p$ is given by \eqref{pij} and
updated in each iteration by \eqref{til-p}. From \eqref{til-p}, it
follows that for each $(i,j)\in U$, $\widetilde p_{ij}\leq p_{ij}$,
the inequality may be true for $(i,j)\in {L \cup \widetilde W_0}$.

(ii) Initially $\widetilde E$ is given by \eqref{tile} and it is
updated at STEP[3] in each iteration. At STEP[3], $F_0 = F_0 \cup L$
and $W_0 = \widetilde W_0 \cup W_0$, by equation \eqref{tile}
$\widetilde E$ reduces if $L \neq \emptyset$ and $\widetilde W_0
\neq \emptyset$ at STEP[3]. if $L = \emptyset$ and $\widetilde W_0 =
\emptyset$, $\widetilde E$ will not be changed by equation
\eqref{tile}.

(iii)   At the start Step[0], we set $r = 0$. Later on we modify r
by \eqref{payoff2}. The matching X satisfies condition $\eqref{X1}$
 In each iteration, this means that $\widetilde{F} \subseteq \widetilde{F}$. Also
 $(new)p \leq (old)p$ by part $(ii)$ of Lemma \ref{la2}. Thus $(new)r_j = f_{ji}(-(new)p_{ij}) \geq
 (old)r_j$. For $j\in (old)\widetilde{F}$,  as matching X
also satisfies $\eqref{X3}$. Moreover, $(new)r_j = (old)r_j = 0$
$\forall j\in F\setminus (new)\widetilde{F}$.  Hence, the vector $r$
increases or remains same.

\end{proof}

\begin{theorem}\label{st}
If the algorithm terminates it produces a stable outcome.
\end{theorem}

\begin{proof}
 We know that $X \subseteq \widetilde E$. Initially $\widetilde E$ is
defined by \eqref{tile} and afterwards it is updated at STEP[3] in
each iteration. Thus $f_{ij}(p_{ij})$ and $f_{ji}(-p_{ij})$ are
non-negative $\forall (i, j) \in \widetilde E$. Therefore,
$f_{ij}(p_{ij}) = 0$ and $f_{ji}(-p_{ij}) = 0$ $\forall (i, j) \in
X$. This shows that the X satisfies (ps1) at termination. On
contrary to (ps2), assume that there exist a $c \in [ a, b]$ and
$(i, j) \in E$ such that

               $f_{ij}(c) > q_i$ and $f_{ji}(-c) > r_j $.

If we take $p_{ij} < c $ it yields $f_{ji}(-p_{ij}) > f_{ji}(-c) >
rj$. But according to Lemma \ref{max-intg}, $p_{ij}$ is the maximum
integer for which this inequality holds. Thus $p_{ij} < c$ is not
true. Let us consider that $p_{ij} = c$, which implies that

\begin{equation}\label{z}
f_{ij}(p_{ij})=f_{ij}(c) > q_i.
\end{equation}
However, at termination we have $U= \emptyset$ means that $(i, j)
\in U$ and since (i, j) are not matched, therefore, $f_{ij}(p_{ij})
< \widetilde q_i = q_i$. A contradiction to \eqref{z}. Thus (ps2)
holds when the algorithm terminates.
\end{proof}

\begin{theorem}\label{termination}
The algorithm terminates after finite number of iterations.
\end{theorem}

\begin{proof}

Termination of the algorithm depends upon set of mutually acceptable
pairs and salary vector $p$. By the Lemma \ref{la2}, part (ii),
$\widetilde E$, reduces either $L \neq \emptyset$, and $\widetilde
W_0, \neq \emptyset$, or remain unchanged. This case is true at most
$|E|$, times.

If $L = \widetilde W_0 = \emptyset$, then, by part (i) of Lemma
\ref{la2}, $p_{ij}$ decreases for each $(i, j) \in U$. Otherwise,
$p$ remains unchanged. As we know that $p$ is bounded and discrete,
therefore, it can be decreased a finite number of times. This proves
that in either case our algorithm terminates after a finite number
of iterations.

This is the most important result which establishes the existence of
pairwise stability for our model.


\end{proof}

\section{Open Problem}
\begin{itemize}
\item We generalized a one-to-one matching model of Ali and Farooq
\cite{ya2010} in two directions. Firstly, our model is many-to-one
model, where as the model given by Ali and Farooq \cite{ya2010} is
one-to-one. Secondly the we the valuations of the players in our
model are represented by strictly increasing general functions of
money. Money is not a continuous variable here. We can establish
similar results by taking money as continuous variable, it will
include a large number of well known models as special cases. A
many-to-many version of our model would be an interesting problem.


\item A slightly harder but worth while problem  would to design a
polynomial time algorithm for these models .
\end{itemize}

\section{Concluding Remarks}
We have represented a many-to-one matching market in which
valuations are represented by general non linear increasing
functions. Each firm can hire the workers according to its demand,
but workers can not work for more than one firms. In this section we
represent some key functions and important remarks about the model
presented in this chapter.
\begin{enumerate}
\item As proved in theorem $(3.6.5)$, the algorithm always produces a
stable many-to-one matching, at termination. This guarantees the
existence of stable outcome for our model.\\

\item The algorithm presented here outputs a worker-optimal
matching. This work is obvious from the result discussed in Lemma
\ref{max-intg}, and the due increasing nature of valuation function.

\item Marriage model by Gale and Shapley \cite{GS}, and model
presented by Ali and Farooq \cite{ya2010}, Ali and Javid model
\cite{ay2014}, are the special cases of this model.
\end{enumerate}



\begin{thebibliography}{99}
%
\bibitem{ya2015} Y. Ali: Stability in a job market
with linearly increasing valuations and quota system. Turk J Math
(2015); {39}: 427-438.
%
\bibitem{ay2014} Y. Ali and A. Javid: Pairwise Stability in Two Sided Market with Strictly
Increasing Valuation Functions. {Discret math theor comput sci}, (To
appear).
%
\bibitem{ya2010} Y. Ali and R. Farooq: Pairwise stability in a two-sided matching market
with indivisible goods and money. J Oper Res Soc Jpn (2011); ${54}$:
1-11.
%
\bibitem{af2011} Y. Ali and R. Farooq: Existence of the
stable outcome for linear valuations and possibly bounded salaries.
Pac J Optim (2011); ${7}$: 531-550.
%
\bibitem{Eco}  F. Echenique and M.B. Yenmez: A solution
to matching with preferences over colleagues. {Games Econ Behav}
(2007); ${59}$:  $46-71$.
%
\bibitem{EriKar}  K. Eriksson and J. Karlander: Stable
matching in a common generalization of the marriage and assignment
Models. Discrete Math (2000); , ${217}$: $135-156$.
%
\bibitem{Far2008} R. Farooq: A polynomial-time algorithm
for a stable matching problem with linear valuations and bounded
side payments. Jpn J Ind Appl Math (2008); ${25}$:  83-98.
%
\bibitem{GS} D. Gale and L.S. Shapley: College admissions and
the stability of marriage. Amer Math Monthly (1962); ${69}$: 9-15.
%
\bibitem{P05} M. Pycia: Many-to-One Matching without Substitutability. MIT Industrial
Performance Center Working Paper 008/2005
%
\bibitem{SS} L. S. Shapley and M. Shubik: The assignment game I: The Core.
Internat J Game Theory (1972); { 1}: 111-130.
%

\bibitem{Sot2000} M. Sotomayor: Existence of stable
outcomes and the lattice property for a unified matching market.
Math Social Sci (2000); {39}:  119-132.
%
\end{thebibliography}
\end{document}